\documentclass[12pt]{amsart}
\usepackage{bbm}
\usepackage[all]{xy}
\newtheorem{theorem}{Theorem}[section]
\newtheorem{lemma}[theorem]{Lemma}
\newtheorem{proposition}[theorem]{Proposition}
\newtheorem{corollary}[theorem]{Corollary}

\theoremstyle{definition}
\newtheorem{definition}[theorem]{Definition}

\theoremstyle{remark}
\newtheorem{remark}[theorem]{Remark}

\numberwithin{equation}{section}

\def\Proj{\mathop\mathrm{Proj}\nolimits}
\def\spec{\mathop\mathrm{Spec}\nolimits}
\def\Tor{\mathop\mathrm{Tor}\nolimits}
\def\Ker{\mathop\mathrm{Ker}\nolimits}
\def\im{\mathop\mathrm{Im}\nolimits}

\def\gcm{\mathop\mathrm{GCM}\nolimits}

\def\de{\delta}		\def\ga{\gamma}
\def\dd{\partial}
\def\al{\alpha}		\def\be{\beta}
\def\vi{\varphi}	\def\tvi{\tilde\varphi}
\def\si{\sigma}		\def\bio{\bar\iota}
\def\La{\Lambda}
\def\De{\Delta}
\def\Om{\Omega}		\def\bom{\overline\Omega}

\def\mP{\mathbb{P}}
\def\gM{\mathfrak m}
\def\gP{\mathfrak p}
\def\dF{\mathfrak F}

\def\kM{\mathcal M}
\def\hM{\overline{M}}
\def\oM{\overline{\mathcal M}}
\def\kW{\mathcal W}
\def\kF{\mathcal F}
\def\vom{\varOmega}
\def\aK{\mathbbm k}
\def\lb{\textup{(}}	\def\rb{\textup{)}}
\def\fY{\mathbf y}	\def\fX{\mathbf x}
\def\fD{\mathbf d}
\def\bK{\mathbf K}
\def\oR{\overline{R}}

\def\bu{_\bullet}
\def\bop{\bigoplus}

\def\+{\oplus}
\def\*{\otimes}
\def\8{\infty}
\def\mt{\mbox{-}}
\def\sb{\subset}	\def\sbe{\subseteq}
\def\xarr{\xrightarrow}
\def\Ar{\Rightarrow}
\def\xx{\times}

\def\MtenL{\mathcal{M} \otimes_{\Lambda} L}

\def\gnr#1{\langle\,#1\,\rangle}

\def\row#1#2{\left( #1_1 , #1_2 , \dots , #1_{#2} \right)}
\def\lst#1#2{ #1_1 , #1_2 , \dots , #1_{#2} }
\def\mtr#1{\begin{pmatrix}#1\end{pmatrix}}
\def\smtr#1{\big(\begin{smallmatrix}#1\end{smallmatrix}\big)}
\def\lsto#1#2{ #1_0 , #1_1 , \dots , #1_{#2} }

\def\cm{Cohen--Macaulay}
\def\iff{if and only if }
\def\acm{arithmetically \cm}

\author[Y. Drozd]{Yuriy A. Drozd}
\author[O. Tovpyha]{Oleksii Tovpyha}
\title[Cohen--Macaulay Rings of Wild Type]{Graded Cohen--Macaulay Rings of Wild Cohen--Macaulay Type}
\address{Institute of Mathematics, National Academy of Sciences, 01601 Kyiv, Ukraine}
\email{y.a.drozd@gmail.com,\,drozd@imath.kiev.ua}
\email{tovpyha@gmail.com}
\urladdr{http://www.imath.kiev.ua/$\sim$drozd}
\subjclass[2010]{13C14, 13H10, 14J60}
\keywords{Cohen--Macaulay rings, Cohen--Macaulay modules, Cohen--Macaulay (representation) type, arithmetically
Cohen--Macaulay varieties}
\date{}

\begin{document}

%Corresponding author: Yuriy Drozd

%Institute of Mathematics, Tereschenkivska 3, 01601 Kyiv, Ukraine

%e-mail: y.a.drozd@gmail.com

%telephone: +380-66-529-0874

%\newpage

\begin{abstract}
We give sufficient conditions for a standard graded \cm\ ring, or equivalently, an arithmetically
\cm\ projective variety, to be \cm\ wild in the sense of representation theory. In particular, these 
conditions are applied to hypersurfaces and complete intersections.
\end{abstract}
\maketitle

%%%%%%%%%%%%%%%%%%%%%%%%%%%%%%%%%%%%%%%%%%%%%%%%%%%%%%%%%%%%%%%%%%%%%%%%%%%%%%%
%%
%% Introduction and Definitions 
%%
%%%%%%%%%%%%%%%%%%%%%%%%%%%%%%%%%%%%%%%%%%%%%%%%%%%%%%%%%%%%%%%%%%%%%%%%%%%%%%%
\section{Introduction}
\label{s1} 

There is a long-standing problem to classify \cm\ rings according the complexity of the classification
of \cm\ modules over them. The known results about such classification gave rise to the conjecture that
all \cm\ rings split into three ``\cm\ types'', namely, \emph{\cm\ discrete} (including \emph{\cm\ finite}), 
\emph{\cm\ tame} and \emph{\cm\ wild}. It is expected to hold in general, though only proved, completely or partially, in several special cases (cf. \cite{jac,dr,gk,dg} for one-dimensional, \cite{en,aus,dgk} 
for two-dimensional, \cite{eh} for graded cases). 

In this paper we give some sufficient conditions for a graded \cm\ ring to be \cm\ wild. Equivalently,
it can be applied to \emph{\acm} (\emph{ACM}) projective varieties $X \subset {\mathbb{P}^n}$
and \emph{ACM sheaves}, giving conditions for such a variety to be \acm\ wild. In particular, we apply
the obtained results to hypersurfaces and complete intersections. 
Note that we consider the ``\emph{algebraical wildness}'', contrary to the viewpoint of \cite{ch,mr,mrpl},
where the wildness means ``\emph{geometrical wildness}'' (we explain the difference below). Note also that,
if $\dim X>1$, \acm\ finiteness (tameness, wildness) is not an invariant of $X$, but of the embedding
$X\hookrightarrow\mP^n$.

We fix a field $\aK$, which is supposed to be infinite, but not necessarily algebraically closed 
or of characteristic $0$.
All algebras are $\aK$-algebras. A \emph{standard graded algebra} is a graded $\aK$-algebra 
$R=\bop_{i=0}^\8R_i$ such that $R_0=\aK$ and $R$ is generated in degree $1$, i.e. $R=\aK[R_1]$.
We denote by $\gM=\bop_{i=1}^\8$ the maximal graded ideal in $R$. We always
suppose that $R$ is \cm\ and denote by $\gcm(R)$ the category of graded maximal \cm\ $R$-modules. 
If we consider $R$ as a homogeneous coordinate ring of a projective variety $X=\Proj R$, then $X$ is
an \emph{\acm\ \lb ACM\rb\ variety} and the category $\gcm(R)$ is equivalent to the category of 
\emph{\acm\ \lb ACM\rb\ coherent sheaves} over $X$ \cite{ch}. If $R$ is an \emph{isolated singularity}, i.e.
all rings $R_\gP$, where $\gP\in\spec R\setminus\{\gM\}$, are regular, then $X$ is a regular variety
and any \emph{ACM} coherent sheaf over $X$ is locally free (or a vector bundle). 

Recall definitions of \emph{\cm\ types} of \cm\ algebras (in graded case).

\begin{definition}\label{11} 
 Let $R$ be a graded \cm\ algebra. It is said to be
 \begin{itemize}
 \item \emph{\cm\ finite} (\emph{CM-finite}) if there is only a finite number of indecomposable
 graded \cm\ $R$-modules (up to shift and isomorphism); otherwise it is said to be \emph{\cm\ infinite}
 (\emph{CM-infinite});
 
 \item \emph{\cm\ discrete} (\emph{CM-discrete}) if, for any fixed rank $r$, there are only finitely
 many \cm\ $R$-modules (up to shift and isomorphism).\!%
 \footnote{\,Sometimes they say \emph{countable} instead of \emph{discrete}, since in this case there is
 only a countable set of \cm\ $R$-modules (up to isomorphism). Nevertheless, it is not very felicitous, since
 if the field $\aK$ is countable it is always the case.}
 \end{itemize}
\end{definition}

 To give definitions of other \cm\ types we need the notion of \emph{families} of \cm\ modules
 analogous to that used in the representation theory of algebras (cf. \cite{d2}).

\begin{definition}\label{12} 
\begin{enumerate}
\item Let $\La$ be a $\aK$-algebra (not necessarily commutative). We consider $R\*\La$ as a graded
algebra setting $(R\*\La)_i=R_i\*\La$. A \emph{family of graded \lb maximal\rb\ 
\cm\ $R$-modules over $\La$} is a finitely generated graded $R\mt\La$-bimodule $\kM$ such that $\kM$
is flat over $\La$ and for every finite dimensional $\La$-module $L$ the $R$-module $\MtenL$ is (maximal) \cm\
over $R$.

We say that the $R$-modules isomorphic to $\MtenL$ \emph{belong to the family} $\kM$.

\item A family $\kM$ of \cm\ $R$-modules over $\La$ is said to be \emph{strict} if, for any finite
dimensional $\La$-modules $L$ and $L'$, 
 \begin{enumerate}
 \item $\MtenL\simeq\MtenL'$ implies $L\simeq L'$;
 \item if $L$ is indecomposable, so is $\MtenL$.
 \end{enumerate} 
\end{enumerate}
If $\kM$ is strict and $\La=\aK[Y]$, where $Y$ is an affine algebraic variety over $\aK$ of dimension $n$, 
we say that $\kM$ is an \emph{$n$-parameter family of graded \cm\ $R$-modules}.
Note that even in this case our definition does not coincide with the ``usual'' definition of a flat 
family of $R$-modules with the base $Y$. Obviously, a family in our sense is also a flat family with the base $Y$, but we do not know whether (or when) the contrary is true.
\end{definition}

\begin{definition}\label{13} 
 A graded \cm\ algebra $R$ is said to be
 \begin{enumerate}
 \item \emph{strictly \cm\ infinite} (\emph{strictly CM-infinite}) if it has a strict family of graded
 \cm\ modules over the polynomial algebra $\aK[x]$.\!%
 \footnote{\,Obviously, in this case there are infinitely many ranks such that there are infinitely many 
 indecomposable \cm\ modules of this rank.}

 \item \emph{\cm\ tame} (\emph{CM-tame}) if, for any fixed rank $r$ there is a finite set $\dF$ of strict
 one-parameter families of graded \cm\ $R$-modules such that for every indecomposable graded \cm\
 $R$-module $M$ of rank $r$ some shift $M(k)$ belongs to a family from $\dF$.
 
 \item \emph{algebraically \cm\ wild} (or briefly \emph{CM-wild}) if for any finitely generated 
 $\aK$-algebra $\La$ there is a strict family of graded \cm\ $R$-modules over $\La$.
 \end{enumerate}
 The last condition means non-formally, that the classification of graded \cm\ $R$-modules is at least
 as difficult as a classification of representations of all finitely generated algebras (commutative or 
 non-commutative).
\end{definition}

If $R$ is the homogeneous coordinate ring of a projective variety $X\sb\mP^n$ and it is CM-finite
(respectively, CM-discrete, CM-tame, strictly CM-infinite or CM-wild), we say that $X$ is \emph{ACM-finite}
(respectively, \emph{ACM-discrete, ACM-tame, strictly ACM-infinite} or \emph{ACM-wild}).

\begin{remark}\label{14} 
\begin{enumerate}
\item  If $R$ is strictly CM-infinite, it cannot be CM-discrete (hence cannot be CM-finite), since
 the modules $\kM\*_{\aK[t]}\aK[t]/(t-a)$, where $a\in\aK$, are indecomposable and of the same rank
 (recall that we suppose the field $\aK$ to be infinite).

\item  It is obvious that if $R$ is algebraically \cm\ wild, it is also \emph{geometrically \cm\ wild} in the 
 sense that for every algebraic variety $Y$ there is a flat family of graded \cm\ $R$-modules with the
 base $Y$ such that all modules in this family are indecomposable and pairwise non-isomorphic. It is not
 known (though conjectured) whether geometrical wildness implies the algebraical one. For projective curves
 (that is for two-dimensional standard graded \cm\ algebras) it follows from the criterion of wildness proved 
 in \cite{dg}. There is a lot of examples of algebras known to be geometrically \cm\ wild such that
 their algebraic wildness has not been proved. The simplest is that of cubic surfaces in $\mP^3$
 \cite{ch}; other examples are in \cite{cmrpl,mr,mrpl}. 

\item  It is known (and easy to see) that in order to show that $R$ is CM-wild it is enough to construct
a strict family of graded \cm\ $R$-modules over \emph{some} algebra $\La$ which is wild in the sense of 
representation theory \cite[Proposition 5.3]{d2}. Examples of such algebras are the free algebra 
$\aK\gnr{x,y}$, the polynomial algebra $\aK[x,y]$ or its localization, the power series algebra 
$\aK[[x,y]]$ \cite{gp}. 
\end{enumerate}
\end{remark}

%%%%%%%%%%%%%%%%%%%%%%%%%%%%%%%%%%%%%%%%%%%%%%%%%%%%%%%%%%%%%%%%%%%%%%%%%%%%%%%
%%
%% The Main Theorem
%%
%%%%%%%%%%%%%%%%%%%%%%%%%%%%%%%%%%%%%%%%%%%%%%%%%%%%%%%%%%%%%%%%%%%%%%%%%%%%%%%
\section{The Main Theorem}
\label{s2} 

From now on $R$ denotes a standard graded \cm\ algebra of Krull dimension $d$.
The following result gives some sufficient conditions for $R$ to be ``bad'' from the point of view
of classification of \cm\ modules.

\begin{theorem}\label{main} 
 Let $\fY=\row yd$ be an $R$-sequence, where $y_i$ is a homogeneous element of degree $m_i>0$,
 $m=\sum_{i=1}^dm_i$ and $\oR=R/\fY$. Consider a homogeneous component $\oR_c$, where $c>m-d+1$.
 \begin{enumerate}
 \item If $\dim \oR_c>1$, then $R$ is strictly CM-infinite.\!%
 \footnote{\,In \cite{st} it has been proved that in this case $R$ is not CM discrete (``countable'').}
 \item If $\dim \oR_c>2$, then $R$ is CM-wild.
 \end{enumerate}
 Equivalently, if $R$ is a homogeneous coordinate ring of a projective variety $X$ in $\mP^n$,
 this variety is, respectively, strictly ACM-infinite or ACM-wild.
\end{theorem}

 The proof of this theorem grounds on the relation between some $\oR$-modules and \cm\ $R$-modules. First we
 prove several lemmas. We always keep the notations and suppositions of Theorem~\ref{main}.

%%%%%%%%%%%%%%%%%%%%%%%%%%%%%%%%%%%%%%%%%%%%%%%%%%%%%%%%%%%%%%%%%%%%%%%%%%%%%%
%% Lemmas
%%%%%%%%%%%%%%%%%%%%%%%%%%%%%%%%%%%%%%%%%%%%%%%%%%%%%%%%%%%%%%%%%%%%%%%%%%%%%%

\begin{lemma}\label{22} 
Let $F$ be a free graded $R$-module, $L$ be a finitely generated graded $R$-module, $\vi:L \to F$ 
be such a homomorphism that the induced map $\bar{\vi}: L/\gM L \to F/\gM F$ is injective.
Then $\vi$ is a split monomorphism, i.e. there is a homomorphism $\vi':F\to L$ such that $\vi'\vi=1_L$.
Thus $F=\im\vi\+F'$, where $F'$ is also free.
\end{lemma}
\begin{proof}
 There is a homomorphism $\bar\pi:F/\gM F\to L/\gM L$ such that $\bar\pi\bar\vi=1_{L/\gM L}$. It can be
 lifted to a homomorphism $\pi:F\to L$ such that $\im(1-\pi\vi)\sbe\gM L$. Then $\pi\vi$ is invertible,
 so $\vi$ is a split monomorphism.
\end{proof}

We denote by $\Om^i(M)$ the \emph{$i$-th syzygy} of the (graded) $R$-module $M$, i.e. the kernel of
the map $\de_{i-1}$ (or, the same, $\im \de_i$), where
\[
 (F\bu,\de\bu):\quad ...\xarr{\de_{i+1}} F_i\xarr{\de_i}F_{i-1}\xarr{\de_{i-1}}\dots 
 \xarr{\de_2} F_1\xarr{\de_1} F_0\xarr{\de_0} M\to0
\]
is the minimal graded free resolution of $M$. Note that $\Om^d(M)$ is always a \cm\ $R$-module. 
We also set $\hM=M\*_R\oR=M/\fY M$, in particular, $\bom^i(M)=\Om^i(M)\*_R\oR$.

\begin{lemma}\label{23} 
Let $W$ be a subspace of $n\oR_c$, $\gnr{W}$ be the $R$-submodule generated by $W$ and $M=n\oR/\gnr{W}$.
Then $M(-m)$ is isomorphic to the submodule of $\bom^d(M)$ generated by the elements of degree $m$.
\end{lemma}

\begin{corollary}[cf. {\cite[Theorem A]{eh}}] \label{24}
Let $W$ and $W'$ are two subspaces of $n\oR_c$, $M=n\oR/\gnr{W}$, $M'=n\oR/\gnr{W'}$. 
The following conditions are equivalent:
 \begin{enumerate}
 \item $M\simeq M'$ (as graded $R$-modules).
 \item $\Om^d(M)\simeq\Om^d(M')$.
 \item There is an automorphism $\si$ of $n\oR$ such that $\si(W)=W'$.
 \end{enumerate}
\end{corollary}
 Moreover, $M$ is indecomposable \iff so is $\Om^d(M)$.
\begin{proof}
 (1)$\Ar$(2) and (3)$\Ar$(1) are trivial.
 
 (2)$\Ar$(1) follows from Lemma~\ref{23}, just as the statement about indecomposability.
 
 (1)$\Ar$(3). Any isomorphism $f:M\to M'$ can be lifted to a commutative diagram
\[
 \xymatrix{
 0 \ar[r] & \gnr{W} \ar[r]\ar[d] & n\oR \ar[r]\ar[d]^{\si} & M \ar[r]\ar[d]^f & 0 \\
 0 \ar[r] & \gnr{W'} \ar[r] & n\oR \ar[r] & M' \ar[r] & 0,
 }
\]
where $\si$ is also an isomorphism. Then $\si(W)=W'$
 since $W=\gnr{W}_c$ and $W'=\gnr{W'}_c$.
\end{proof}

Before proving Lemma~\ref{23} we establish the following fact. We keep the conditions of this Lemma and
denote by
%\begin{itemize}
%\item[] 

 $(K\bu,\dd\bu)$ the Koszul resolution of $\oR=R/\fY$;
%\item[] 
 
 $(F\bu,\de\bu)$ a minimal graded free resolution of $M$;
%\item[] 

$(\bK\bu,\fD\bu)=n(K\bu,\dd\bu)$, the $n$-fold Koszul complex, which is a minimal 
 graded free resolution of $n\oR$;
%\item[] 

$\vi\bu:(\bK\bu,\fD\bu)\to(F\bu,\de\bu)$ a morphism of complexes induced by the
epimorphism $n\oR\to M$.
%\end{itemize}

\begin{lemma}\label{25} 
 The map $\vi_i$ is a split monomorphism for each $i$ and $F_i=\im\vi_i\+F'_i$, where $F'_i$
 contains no homogeneous elements of degrees less than $c+i-1>m$.
\end{lemma}
\begin{proof}
 We use induction on $i$. For $i=0,1$ it follows from the definition of $(F\bu,\de\bu)$. Suppose that
 the claim is true for all $i\le j$. We identify $\im\vi_i$ with $\bK_i$. Since $\gM F'_{j-1}$ contains
 no elements of degree less than $c+j-2>m$, while $\bK_j$ is generated by elements of degree at most $m$,
 the matrix presentation of the homomorphism $\de_j:\bK_j\+F'_j\to\bK_{j-1}\+F'_{j-1}$ is of the form
 \[
   \de_j=\mtr{\fD_j&\be_j\\0&\ga_j}
 \]
 Let $\de_{j+1}:F_{j+1}\to\bK_j\+F'_j$ be of the form $\smtr{\al\\\be}$. Then $\fD_j\al+\be_j\be=0$.
 Consider any element $u\in F_{j+1}$ of degree $k<c+j$. Then $\be(u)=0$ since there are no elements 
 of this degree in $\gM F'_j$, hence $\fD_j\al(u)=0$ and $\al(u)\in\Ker\fD_j=\im\fD_{j+1}$. So 
 $\al(u)=\fD_{j+1}(v)=\al\vi_{j+1}(v)$ for some $v\in\bK_{j+1}$ and 
 $u-\vi_{j+1}(v)\in\Ker \de_{j+1}\sbe\gM F_{j+1}$. Therefore $F_{j+1}=\im\vi_{j+1}+\tilde F$,
 where $\tilde F$ is generated by elements of degrees at least $c+j$. Since 
 $\de_{j+1}\vi_{j+1}=\vi_j\fD_{j+1}$, $\de_{j+1}(\im\vi_{j+1})\sbe\fY F_j$. Note also that 
 $\be\vi_{j+1}=0$, since $\bK_{j+1}$ is generated by elements of degrees at most $m$.
 
 Now consider the maps $\tvi_i:\bK_i/\gM\bK_i\to F_i/\gM F_i$ induced by $\vi_i$. They are injective
 for $i\le j$. Let $v\notin\gM\bK_{j+1}$ be such that $\vi_{j+1}(v)\in\gM F_j$. Then 
 $\vi_{j+1}(v)\in\gM(\im\vi_{j+1})$, since all elements of $\tilde F$ are of degrees greater than $m$
 and $\deg v\le m$. Thus $\vi_j\fD_{j+1}(v)=\de_{j+1}\vi_{j+1}(v)\in\fY\gM F_j$. It is impossible,
 since $\fD_{j+1}(v)\notin\gM\fY\bK_j$ and $\vi_j$ is a split monomorphism. Therefore, $\tvi_{j+1}$
 is a monomorphism and $\vi_{j+1}$ is a split monomorphism by Lemma~\ref{22}.
\end{proof}

\begin{proof}[Proof of Lemma~\ref{23}]
 Consider the exact sequence
 \[
  0\to \Om^d(M)\xarr{\iota} F_{d-1}\to \Om^{d-1}(M)\to 0. 
 \]
 Tensoring it with $\oR$ we get an exact sequence
\[
 0\to \Ker\bio\to \bom^d(M)\xarr{\bio} \bar F_d\to \bio^{d-1}(M)\to 0,
\]
 where $\Ker\bio\simeq \Tor_1^R(\Om^{d-1}(M),\oR)\simeq\Tor^R_d(M,\oR)$. On the other hand,
\[
 \Tor_d^R(M,\oR)\simeq\Ker\{1\*\dd_d:M\*_R K_d\to M\*_R K_{d-1}\}\simeq M(-m),
\]
 since $K_d\simeq R(-m)$, $\im\dd_d\sbe\fY K_{d-1}$ and $\fY M=0$. Hence $\Ker\bio$ is generated by
 elements of degree $m$. On the other hand, all elements of degree $m$ in $\Om^d(M)$ are in $\im\fD_d$,
 hence in $\fY F_{d-1}$. Therefore, all elements of degree $m$ in $\bom^d(M)$ are in $\Ker\bio$.
 It means that $\Ker\bio\simeq M(-m)$ coincides with the submodule of $\bom^d(M)$ generated by
 elements of degree $m$.
\end{proof}

 As we use $R\mt\La$-bimodules, or, the same, $R\*\La$-modules, we also need the following 
 fact about projective resolutions, which follows from the Grothendieck's Generic 
 Freeness Lemma \cite[Theorem 14.4]{ei}.
 
\begin{proposition}\label{26} 
 Let $\La$ be a finitely generated commutative algebra without zero divisors, $\kM$ be a 
 finitely generated $R\*\La$-module and $k$ be a positive integer. There is a non-zero element 
 $a\in\La$ such that there is a free resolution of $\kM[a^{-1}]$ as of $R\*\La[a^{-1}]$-module
 \begin{equation}\label{e21} 
 \dots\to\kF_k \xarr{\De_k} \kF_{k-1}\xarr{\De_{k-1}}\dots 
 \to \kF_1\xarr{\De_1}\kF_0\xarr{\De_0}\kM[a^{-1}]\to 0
 \end{equation}
 such that $\kM[a^{-1}]$ and $\Ker\De_i$ are free as $\La[a^{-1}]$-module and
 $\Ker\De_i\sbe\gM\kF_i$ for all $i\le k$.
 \end{proposition} 
 \begin{proof}
  We use induction by $r$. By the Grothendieck's Generic Freeness Lemma there is a non-zero $a\in\La$
  such that $\kM[a^{-1}]$ and all homogeneous components of $\oM=\kM[a^{-1}]/\gM\kM[a^{-1}]$ 
  are free over $\La'=\La[a^{-1}]$. 
  Set $\kF_0=\bop_iR[-i]\*\oM_i$. It is a free $R\*\La'$-module and there is an epimorphism
  $\De_0:\kF_0\to\kM[a^{-1}]$ such that $\Ker\De_0\sbe\gM\kF_0$. Obviously, $\Ker\De_0$ is also free
  over $\La'$. It proves the statement for $k=0$.
  Suppose that we already have found $a$ and a resolution \eqref{e21} for given $k$. Applying the same
  procedure to $\Ker\De_k$, we obtain an analogous resolution for $k+1$.
  \end{proof}

\begin{proof}[Proof of Theorem~\ref{main}]
 We prove the statement (2), since the proof of (1) is quite analogous. Namely, we construct a
 family of graded \cm\ $R$-modules over a localization of the polynomial ring $\La=\aK[x,y]$.
 
 Suppose that $\dim_\aK \oR_c>2$ and $e_1,e_2,e_3$ are linear independent elements from $\oR_c$.
 In the $R\*\La$-module $\oR\*\La$ consider the $\La$-submodule
 $\kW$ generated by $e_1\*1+e_2\*x+e_3*y$. Note that $\kW\simeq\La$ as $\La$-module. Let 
 $\kM=\oR\*\La/\gnr{\kW}$, where $\gnr{\kW}$ is the $\oR\*\La$-submodule generated by $\kW$.
 Applying Proposition~\ref{26}, we find a non-zero element $a\in\La$ and a free resolution 
 of the shape \eqref{e21} for $\kM'=\kM[a^{-1}]$ over $R\*\La$, where $\La'=\La[a^{-1}]$.
 Let $\vom=\Ker\De_{d-1}=\im\De_d$. We prove that $\vom$ is a strict family of graded \cm\
 $R$-modules over $\La'$. The exact sequence \eqref{e21} splits as a sequence of $\La'$-modules.
 Hence for any $\La'$-module $L$ the sequence
 \begin{multline*}
  \dots\to\kF_d\*_{\La'} L \xarr{\De_d\*1} \kF_{d-1}\*_{\La'} L\xarr{\De_{d-1}\*1}\dots\to \\
 \to \kF_1\*_{\La'} L\xarr{\De_1\*1}\kF_0\*_{\La'} L\xarr{\De_0\*1}\kM'\*_{\La'} L\to 0 
 \end{multline*}
 is exact, so is a minimal resolution of the $R$-module $\kM'\*_\De L$. In particular,
 $\vom\*_\De L\simeq\Om^d(\kM'\*_\De L)$ is a \cm\ $R$-module.  
  
 A finite dimensional $\La'$-module $L$ is given by two square $n\xx n$ matrices $A_x,A_y$ over $\aK$,
 where $n=\dim_\aK L$, describing respectively the action of $x$ and $y$ in a basis of $L$. Then
 $\kM'\*_{\La'} L\simeq n\oR/\gnr{\kW\*_{\La'}L}$ and $\gnr{\kW\*_{\La'}L}$ is the subspace of
 $n\oR_c$ generated by the columns of the matrix $Ie_1+A_xe_2+A_ye_3$, where $I$ is the $n\xx n$
 unit matrix. Suppose that $L'$ is another $La'$-modules of the same dimension given by a pair of
 matrices $B_x,B_y$ and $\kM'\*_{\La'} L\simeq\kM'\*_{\La'} L'$. By Lemma~\ref{23}, there is an
 automorphism $\si$ of $n\oR$ such that $\si(\kW\*_{\La'}L)=\kW\*_{\La'}L'$. This automorphism can
 be considered as invertible $n\xx n$ matrix over $\aK$ and $\si(\kW\*_{\La'}L)=\kW\*_{\La'}L'$
 means that there is an invertible matrix $\tau$ such that 
 $\si(Ie_1+A_xe_2+A_ye_3)=(Ie_1+B_xe_2+B_ye_3)\tau$. As the elements $e_1,e_2,e_3$ are linear independent,
 we obtain that $\si=\tau$, $B_x=\si A_x\si^{-1}$ and $B_y=\si A_y\si^{-1}$, hence $L\simeq L'$.
 Analogously one proves that if $L$ is indecomposable, $\vom\*_{\La'}L$ is indecomposable as well.
\end{proof}

%%%%%%%%%%%%%%%%%%%%%%%%%%%%%%%%%%%%%%%%%%%%%%%%%%%%%%%%%%%%%%%%%%%%%%%%%%%%%%%
%%
%% Corollaries
%%
%%%%%%%%%%%%%%%%%%%%%%%%%%%%%%%%%%%%%%%%%%%%%%%%%%%%%%%%%%%%%%%%%%%%%%%%%%%%%%%
\section{Applications}
\label{s3} 

 Theorem~\ref{main} implies wildness conditions for special types of algebras (or varieties).
 In what follows we write $\aK[\fX]=\aK[\lsto xn]$.
 
\begin{corollary}\label{31} 
 A hypersurface of degree $e\ge4$ in $\mP^n$ is strictly ACM-infinite. If $n\ge2$, it is ACM-wild.
\end{corollary}
\begin{proof}
 Consider the homogeneous coordinate ring $R=\aK[\fX]/(f)$ of $X$, where $\deg f=e$. It is of
 Krull dimension $n$. As $\aK$ is infinite, we can suppose that $\row xn$ is an $R$-sequence. Set first
 $y_1=x_1^2$ and, if $n>1$, $y_i=x_i$ for $2\le i\le n$. Then $m-d+2=3$ and $\dim_\aK\oR_3=2$, so $R$
 is strictly CM-infinite. 
 
 If $n\ge2$, set $y_1=x_1^2,\,y_2=x_2^2$ and, if $n>2$, $y_i=x_i$ for $3\le i\le n$. Then $m-d+2=4$
 and $\dim_\aK\oR_4\ge3$. Hence $R$ is CM-wild.
\end{proof}

Note that similar result for local rings was obtained by V.\,V.\,Bon\-da\-renko \cite{b} 
using the technique of matrix factorizations. For graded rings it has also been proved in \cite{cl}.

%%%%%%%%%%%%%%%%%%%%%%%%%%%%%%%%%%%%%%%%%%%%%%%%%%%%%%%%%%%%%%%%%%%%%%%%%%%%%%%

\begin{corollary}\label{32} 
 Let $X\sb\mP^n$ be a complete intersection in $\mP^n$ of codimenion $k>1$ defined by polynomials
 $\lst fk$ of degrees $\deg f_j=d_j>1$. If $k\ge3$ or $d_1\ge3$, then $X$ is ACM-wild.
\end{corollary}
\begin{proof}
 We consider again the homogeneous coordinate ring $R$ of $X$, $R=\aK[\fX]/\row fk$,
  and suppose that
 $x_k,x_{k+1},\dots,x_n$ is an $R$-sequence. Set $y_1=x_k^2$ and, if $n>k$, $y_i=x_{k+i-1}$ for 
 $1<i\le n-k+1$. Then $m-d+2=3$ and one easily check that $\dim\oR_3\ge4$ if either $k\ge3$ or 
 $d_1\ge3$. Therefore, $R$ is CM-wild. 
\end{proof}

 Note that cubic surfaces in $\mP^3$ as well as intersections of two quadrics in $\mP^4$ are known to be 
 geometrically \cm\ wild \cite{ch,mrpl}. The proof that they are algebraically \cm\ wild, as well as
 their analogues of higher dimensions, is an intriguing problem. Actually, in \cite{ch,mrpl} it is proved
 even more: there are families of non-isomorphic indecomposable \emph{Ulrich} bundles of arbitrary big
 dimensions. In \cite{bdiss} it is proved that a \emph{local} algebra $\aK[[\lsto xn]]/(f)$ is CM-wild if
 $n\ge3$ and $f$ does not contain linear and quadratic terms.

%%%%%%%%%%%%%%%%%%%%%%%%%%%%%%%%%%%%%%%%%%%%%%%%%%%%%%%%%%%%%%%%%%%%%%%%%%%%%%%
%%
%% References
%%
%%%%%%%%%%%%%%%%%%%%%%%%%%%%%%%%%%%%%%%%%%%%%%%%%%%%%%%%%%%%%%%%%%%%%%%%%%%%%%%

\end{document}